\documentclass{amsart}

\input xy
\xyoption{all}

\newbox\noforkbox \newdimen\forklinewidth
\setbox0\hbox{$\textstyle\smile$}
\setbox1\hbox to \wd0{\hfil\vrule width \forklinewidth depth-2pt
 height 10pt \hfil}
\setbox\noforkbox\hbox{\lower 2pt\box1\lower 2pt\box0\relax}
\def\unionstick{\mathop{\copy\noforkbox}\limits}
\def\nonfork_#1{\unionstick_{\textstyle #1}}

\setbox0\hbox{$\textstyle\smile$}
\setbox1\hbox to \wd0{\hfil{\sl /\/}\hfil}
\setbox2\hbox to \wd0{\hfil\vrule height 10pt depth -2pt width
              \forklinewidth\hfil}
\newbox\doesforkbox
\setbox\doesforkbox\hbox{\lower 2pt\box1 \lower 2pt\box2\lower2pt\box0\relax}
\def\nunionstick{\mathop{\copy\doesforkbox}\limits}

\def\fork_#1{\nunionstick_{\textstyle #1}}

\newbox\noforkboxs \newdimen\forklinewidths
\setbox0\hbox{$\textstyle\smile$}
\setbox1\hbox to \wd0{\hfil\vrule width \forklinewidths depth-2pt
 height 10pt \hfil}
\setbox\noforkboxs\hbox{\lower 2pt\box1\lower 2pt\box0\relax}
\def\unionsticks{\mathop{\copy\noforkboxs^{\mkern-8mu\mathrm{s}}}\limits}
\def\nonforks_#1{\unionsticks_{\textstyle #1}}

\setbox0\hbox{$\textstyle\smile$}
\setbox1\hbox to \wd0{\hfil{\sl /\/}\hfil}
\setbox2\hbox to \wd0{\hfil\vrule height 10pt depth -2pt width
              \forklinewidth\hfil}
\newbox\doesforkboxs
\setbox\doesforkboxs\hbox{\lower 2pt\box1 \lower 2pt\box2\lower2pt\box0\relax}
\def\nunionsticks{\mathop{\copy\doesforkbox}\limits}

\def\forks_#1{\nunionsticks_{\textstyle #1}}

\usepackage{amsmath}
\usepackage{amssymb}

\usepackage{amsthm}  

\newtheorem{same}{This should never appear}[section]
\newtheorem{defin}[same]{Definition}

\newtheorem{remark}[same]{Remark}
\newtheorem{theorem}[same]{Theorem}

\newtheorem{lemma}[same]{Lemma}
\newtheorem{fact}[same]{Fact}
\newtheorem{question}[same]{Question}

\newtheorem{prop}[same]{Proposition}

\newtheorem*{theorem-2}{Theorem \ref{sta-2}}
\newtheorem*{theorem-3}{Theorem \ref{equiv}}
\newtheorem*{cor-1}{Corollary \ref{abs}}
\newtheorem*{theorem-4}{Theorem \ref{inj-equiv}}
\newcommand{\Cl}{\operatorname{cl}}

\newtheorem{defin*}{Definition}
\newtheorem*{theorem*}{Theorem}

\makeatletter
\newcommand{\skipitems}[1]{%
  \addtocounter{\@enumctr}{#1}%
}

\newcommand{\rest}{\mathord{\upharpoonright}}
\newcommand{\id}{\operatorname{id}}

\newcommand{\Kk}{\mathbf{K}}

\newcommand{\Kh}{\hat{\mathbf{K}}}

\newcommand{\restr}{\upharpoonright}

\newcommand{\LS}{\operatorname{LS}}

\newcommand{\Z}{\mathbb{Z}}
\newcommand{\mcU}{ U}
\newcommand{\mcV}{ V}
\newcommand{\leap}[1]{\le_{#1}}

\newcommand{\lea}{\leap{\Kk}}

\newcommand{\gtp}{\mathbf{gtp}}
\newcommand{\gS}{\mathbf{gS}}

\usepackage{tikz-cd}
\newcommand{\Q}{\mathbb{Q}}

\title[An unstable AEC of modules: A variation of Paolini-Shelah's example]{An unstable abstract elementary class of modules: A variation of  Paolini-Shelah's example}

\author[Herden, Mazari-Armida, Walton]{Daniel Herden, Marcos Mazari-Armida, Michael D. Walton}

\thanks{The first author was supported by Simons Foundation grant MPS-TSM-00007788.
	The second author was supported by NSF grant DMS-2348881 and Simons Foundation grant MPS-TSM-00007597}

\address{\newline Daniel Herden, Marcos Mazari-Armida, and Michael D. Walton \newline Department of Mathematics \newline Baylor University \newline Sid Richardson Building \newline 1410 S.~4th Street \newline	Waco, TX 76706, USA}
\urladdr{https://sites.baylor.edu/daniel\_herden/}
\urladdr{https://sites.baylor.edu/marcos\_mazari/}

\email{daniel\_herden@baylor.edu; marcos\_mazari@baylor.edu, michael\_walton2@baylor.edu}

\begin{document}

\begin{abstract} 
We construct a class $\hat{K}$ of torsion-free abelian groups  such that $\Kh =(\hat{K}, \leq_p)$ is an abstract elementary class with $\LS(\Kh)=\aleph_0$ such that 
\begin{itemize}
\item $\Kh$ is not stable.
\item $\Kh$ has the joint embedding property and no maximal models, but does not have the amalgamation property.
\item $\Kh$ is $(<\aleph_0)$-tame. 
\end{itemize} 
The  class we construct is a variation of \cite[Section 4]{ps} which isolates the core mechanism of the Paolini-Shelah construction.
\end{abstract}


\maketitle

{\let\thefootnote\relax\footnote{{AMS 2020 Subject Classification:
Primary:  03C60, 03C48. Secondary: 03C45, 20K20,  	20K40.

  }}}


\section{Introduction}
Fisher \cite{fisher} and Baur \cite[Theorem 1]{baur} showed in the seventies  that if $T$ is a complete first-order theory of $R$-modules then $T$ is stable. A natural research direction is to determine if the previous result can be extended beyond first-order model theory. Recently, the focus has been on the following question for abstract elementary classes: 
 
\begin{question}[{\cite[Question 2.12]{maztor}}]\label{mainq} Let $R$ be an associative unital ring and denote the pure submodule relation by $\leq_p$.
If $(K, \leq_p)$ is an abstract elementary class such that $K$ is a class of $R$-modules, is $(K, \leq_p)$ stable? Is this true if $R =\mathbb{Z}$?
\end{question} 

For many years the evidence pointed to a \emph{positive} answer as for example it was shown to be positive for the following classes of modules: torsion-free abelian groups \cite{baldwine}, $R$-modules  \cite{kuma}, torsion abelian groups \cite{maztor}, $\aleph_1$-free abelian groups \cite{maztor}, $R$-flat modules \cite{lrvcell}, $R$-absolutely pure modules \cite{maz2}, and $R$-flat modules of dimension $\leq n$ \cite{mt1}.\footnote{See  \cite{bon25} and \cite[Section 2]{maztor} for more on the positive  direction.} Surprisingly, in December 2025 Paolini and Shelah \cite{ps} constructed a class of torsion-free abelian groups which shows that the answer to Question \ref{mainq} is \emph{negative} even when $R= \mathbb{Z}$.

 The purpose of this paper is to study a variation of the example of Paolini and Shelah \cite[Section 4]{ps}. This is pursued in order to better understand the core mechanism of the Paolini-Shelah construction and to try to use this knowledge to isolate general conditions that imply stability for AECs of modules. 
 
 More precisely, we construct  a class $\hat{K}$ of torsion-free abelian groups (see Definition \ref{class}) such that the following holds:

\begin{theorem}\label{maint}
 $\Kh=(\hat{K}, \leq_p)$ is an abstract elementary class with $\LS(\Kh)=\aleph_0$ such that:
    \begin{enumerate}
    \item  $\Kh$ is not stable. 
        \item  $\Kh$ has the joint embedding property and no maximal models.
        \item  $\Kh$ does not have the amalgamation property.
         \item  $\Kh$ is $(<\aleph_0)$-tame.

    \end{enumerate}
\end{theorem}

 The reason stability fails in our example is essentially the same reason as that of \cite[Section 4]{ps}, but the abstract elementary class we construct around the abelian groups that witness unstability is described by fewer conditions which are moreover easier to check than those of \cite[Section 4]{ps}. More importantly,  our abstract elementary class is likely better behaved than that of \cite{ps}. To be precise, it is unclear if the example of Paolini and Shelah  satisfies Conditions (2) and (4) of  Theorem \ref{maint}.

 Of particular interest to us is that $\Kh$ is $(<\aleph_0)$-tame as one could have hoped that the answer to Question \ref{mainq} would be  affirmative under this strong assumption.  We show that $\Kh$ is tame by showing that it admits intersections and that closures are canonically built (see Proposition \ref{closure} and Theorem \ref{Thm:tame}).

Beyond answering Question \ref{mainq},  we think that the existence of AECs of modules that are not stable provides additional evidence that abstract elementary classes are noticeably more complicated objects than their first-order counterparts. 

The paper is divided into three sections. Section 2 presents necessary
background on  abelian group theory and abstract elementary classes required to read the paper. Section 3 has the main construction and results of the paper. Readers interested solely in unstability essentially only need to look at Definition \ref{class}, Proposition \ref{closure}, Definition \ref{groups}, and Theorem \ref{Prop:K_not_stable}. 

\section{Preliminaries}

We briefly introduce the notions of abelian group theory and abstract elementary classes that will be used in this paper. Further details on abelian group theory can be consulted in \cite{fuc} and on abstract elementary classes in \cite{baldwinbook09}.

\subsection*{Abelian groups} All groups discussed in this paper are abelian groups. Given a group $G$ and $n \in \mathbb{Z}$, we let $nG = \{ng : g\in G\}$.

An abelian group $G$ is \emph{torsion-free} if every $0\ne g \in G$ has infinite order.  We say that $H$ is a \emph{pure subgroup} of $G$, denoted by $H \leq_p G$, if $H$ is a subgroup of $G$ and for every $n\in \Z$, $nH = H\cap nG$. If $G, H$ are torsion-free groups, to check purity it is enough to show that $pH = H\cap pG$ for every prime number $p$.

Given a group $G$ and $p$ a prime number, let $p^\omega G := \bigcap_{n\in \Z_{\geq 0}} p^n G = \{ g \in G : p^n \mid g \text{ for all } n < \omega \}$. We will sometimes write $p^\infty \mid g$ to indicate that $g \in p^\omega G$. We will use the next fact often without reference.

\begin{fact}  Let $G, H$ be torsion-free abelian groups and $H \leq_p G$. 
\begin{enumerate}
\item If $ng \in H$ for $g\in G$ and $n \in \Z_{>0}$, then $g \in H$.
\item  For every prime number $p$, $p^\omega H=p^\omega G \cap H$. In particular, if $p^\omega G = 0$ then $p^\omega H = 0$.
\end{enumerate}
\end{fact}

 Finally, given prime numbers $p, q, r$,  let $\Z[1/p, 1/q, 1/r]$ denote the subring of $\Q$ that results from adjoining the rational number $1/p, 1/q, 1/r$ to $\Z$, i.e.,  \[\Z[1/p, 1/q, 1/r] = \{ p^{-n}a + q^{-m}b + r^{-s}c  : n, m,s \in \Z_{\ge 0}, a, b, c \in \Z\}.\]

 \subsection*{Abstract elementary classes} Abstract elementary classes (or AECs) were introduced by Shelah in the seventies \cite{sh88}. An \emph{abstract elementary class} is a pair $\Kk = (K, \lea)$ where $K$ is a class of $L$-structures (for some fixed finitary language $L = L(\Kk)$)\footnote{In this paper, $L$ will be the language of abelian groups, i.e., $L= \{ +, -, 0\}$  where $+,-$ are binary functions.} and $\lea$ is a partial order on $K$ extending the substructure relation\footnote{In this paper, $\lea$ will be the pure subgroup relation $\leq_p$.} such that $\Kk$ is closed under isomorphisms and increasing continuous $\lea$-chains; and satisfies coherence and  a version of the downward L\"owenheim-Skolem theorem. See \cite[Definition 4.1]{baldwinbook09} for the full definition.
 
 Given $M \in K$, we will write $|M|$ for the underlying set of the model and $\|M\|$ for its cardinality. We say that $f: M \to N$ is a \emph{$\Kk$-embedding} if $f: M \cong f[M] \lea N$. In this paper all $\Kk$-embeddings will be pure embeddings. 
 
 We will investigate the following three properties in this paper.
 
 \begin{defin} Let $\Kk=(K, \lea)$ be an AEC.
  \begin{enumerate}
        \item $\Kk$ has \emph{no maximal models} if for every $M\in K$, there exists an $N\in K$ with $M \lea N$ and $|M| \subsetneq |N|$.

        \item $\Kk$ has the \emph{joint embedding property} if for every  $M_1, M_2 \in K$, there exist $N \in K$ and $\Kk$-embeddings $f_1: M_1 \to N$ and $f_2: M_2 \to N$. 

        \item $\Kk$ has the \emph{amalgamation property} if every span of $\Kk$-embeddings $N_1 \leftarrow M \to N_2$ can be completed to a commutative square of $\Kk$-embeddings.
 \end{enumerate}
 \end{defin}

The main example of this paper has the additional property of admitting intersections. This is a strong property that often fails even for elementary classes. These AECs were introduced in \cite{bash} and further studied in \cite[Section 2]{vaseyu}.

\begin{defin} An AEC $\Kk$ \emph{admits intersections} if for every $N \in K$ and $A \subseteq |N|$,  $\Cl^{N}_{\Kk}(A):= \bigcap\{M :  A \subseteq M \lea N \text{and } M \in K\} \in K$ and $\Cl^{N}_{\Kk}(A) \lea N$. 
\end{defin}

Galois types are the correct generalization of first-order types to AECs and were first introduced by Shelah in \cite{sh300}. We introduce Galois types for AECs that admit intersections. 

\begin{defin} Let $\Kk$ be an AEC that admits intersections and $\Kk^3$ be the set of triples of the form $(b, A, N)$, where $N \in K$, $A \subseteq |N|$, and $b \in N$.

\begin{enumerate}
    \item For $(b_1, A_1, N_1), (b_2, A_2, N_2) \in \Kk^3$, we say that $(b_1, A_1, N_1)E (b_2, A_2, N_2)$ if $A := A_1 = A_2$, and there exists $f : \Cl_{\Kk}^{N_1}(\{ b_1\} \cup A) \cong \Cl_{\Kk}^{N_2}(\{ b_2\} \cup A)$ such that $f \rest A =\id_A$ and $f_1 (b_1)=b_2$. It is easy to check that $E$ is an equivalence relation.
    \item For $(b, A, N) \in \Kk^3$, the \emph{Galois type} of $b$ over $A$ in $N$, denoted by $\gtp_{\Kk} (b / A; N)$, is the $E$-equivalence class of $(b, A, N)$.

  \end{enumerate}
\end{defin}

\begin{remark}
The definition above is equivalent to the standard definition of Galois type by \cite[Proposition 2.18]{vaseyu}.
\end{remark}

Since types are semantic objects, they might not be determined by their restriction to finite subsets. In case that they are determined by their restrictions to finite subsets we say that $\Kk$ is \emph{$(<\aleph_0)$-tame}. More precisely, $\Kk$ is \emph{$(<\aleph_0)$-tame}  if for every $M \in K$ and $\gtp_{\Kk} (b_1 / M; N_1) \ne \gtp_{\Kk} (b_2 / M; N_2)$, there is $A \subseteq |M|$ such that $|A|< \aleph_0$ and $\gtp_{\Kk} (b_1 / A; N_1) \ne \gtp_{\Kk} (b_2 / A; N_2)$.

A fundamental topic of study in model theory, and in the theory of AECs in particular, are dividing lines \cite{sh-div}. In this paper, we will study the dividing line of stability. 

\begin{defin}
$\Kk$ is \emph{stable in $\lambda$} if  for any $M \in K$ with $\|M\|=\lambda$ it holds that $| \gS_{\Kk}(M) | \leq \lambda$, where $\gS_{\Kk}(M) = \{ \gtp_{\Kk}(a/M; N) : M \lea N \text{ and } a \in N\}$.

We say that $\Kk$ is \emph{stable} if there exists a  cardinal $\lambda \geq \LS(\Kk)$ for which $\Kk$ is stable in $\lambda$.
\end{defin}

 These are all the notions of AECs used in this paper. Detailed introductions to abstract elementary classes from an algebraic perspective are given in \cite{bon25} and \cite[Section 2]{maztor}.

\section{Main Results}

 We introduce the main object of study of this paper.

\begin{defin}\label{class}
    Let $\bar{p} = (p_1, p_3, p_4)$ be distinct primes.\footnote{The reason we do not pick a prime $p_2$ is so our example can be easily compared to \cite{ps}.}

    Let $\hat{K} = \hat{K}(\bar{p})$ be the class of all torsion-free abelian groups $G$ such that:
    \begin{enumerate}
        \item $G = p_1^\omega G$ and $p^\omega G = 0$ for all primes $p\neq p_1$, \textbf{or}
        \item 
        \begin{enumerate}
            \item for all $g\in p_1^\omega G$, there exists \emph{at most} one $z\in p_3^\omega G$ such that $g+z\in p_4^\omega G$, and
            \item for all $g\in p_1^\omega G$, there are $k\in \Z_{>0}$ and $z\in p_3^\omega G$ with $kg + z\in p_4^\omega G$.
        \end{enumerate}
    \end{enumerate}

For $\ell \in \{1,2\}$, if a torsion-free abelian group $G$ satisfies Condition $(\ell)$ we will say that $G \in \hat{K}_\ell$. Let $\Kh=(\hat{K}, \leq_p)$ and $\Kh_\ell=(\hat{K}_\ell, \leq_p)$.
\end{defin}

\begin{remark}
    As mentioned in the introduction, the class above is a variation of \cite[Section 4]{ps}. The classes are not formally comparable, but  Condition (1) of Definition \ref{class} loosely corresponds to Condition ($\star_1$)(c) of \cite[Section 4]{ps} and Conditions (2)(a) and (2)(b) of Definition \ref{class} loosely correspond to
    Conditions ($\star_1$)(d)$(\cdot_5)$ and ($\star_1$)(d)$(\cdot_6)$ of \cite[Section 4]{ps}, respectively.\footnote{Note however, as a marked difference, that for $H_1:= \{g\in p_1^\omega G: \mbox{there exists some } z\in p_3^\omega G$ such that $g+z\in p_4^\omega G\}$ Condition ($\star_1$)(d)$(\cdot_6)$ requires $p_1^\omega G/H_1$ to be torsion-free, while for our class $p_1^\omega G/H_1$ is torsion by Condition (2)(b).}\end{remark}

We begin with some basic observations.

\begin{prop}\label{basic}\
\begin{enumerate}
\item $\hat{K}_1 \cap  \hat{K}_2 = \{ 0\}$.
\item If $G \in \hat{K}_1$ and $H \leq_p G$, then $H \in \hat{K}_1$.
\item If $G \in \hat{K}_2$ with $G\ne 0$, $H \in \hat{K}$, and $G \leq_p H$, then $H \in \hat{K}_2$.
\end{enumerate}
\end{prop}
\begin{proof} \ 
\begin{enumerate}
    \item Let $G\in \hat{K}_1 \cap \hat{K}_2$. As $G\in \hat{K}_1$, $G = p_1^\omega G$ and $p_3^\omega G = p_4^\omega G = 0$. Since $G\in \hat{K}_2$, for all $g\in G$, there exists $k\in \Z_{>0}$ and $z=0 \in p_3^\omega G$ such that $kg+0 = 0 \in p_4^\omega G$. But since $G$ is torsion-free and $kg=0$ with $k\neq 0$, it follows that $g=0$, hence $G=0$.
    \item Observe that $p_1^\omega H = H\cap p_1^\omega G = H\cap G=H$ and $p^\omega H = H\cap p^\omega G = H\cap 0=0$ for $p \neq p_1$. Hence $H\in \hat{K}_1$.
    \item Since $G \not\in \hat{K}_1$, either  $G \neq p_1^\omega G$ or $p^\omega G \neq 0$ for some prime $p\neq p_1$. Since $G \leq_p H$, the failure of any of these properties transfers to $H$. Hence $H \not\in \hat{K}_1$, so $H \in \hat{K}_2$.\qedhere
\end{enumerate}

\end{proof}

We next show that $\Kh$ is indeed an AEC.

\begin{lemma}\label{kh-aec}
$\Kh=(\hat{K}(\bar{p}), \leq_p)$ is an AEC with $\LS(\Kh)=\aleph_0$.
\end{lemma}
\begin{proof} From our discussion of $\Cl_{\Kh}^G(A)$ in Proposition \ref{closure}, $\LS(\Kh)=\aleph_0$ will be clear.  
So, at this point, we only need to check 
 the Tarski-Vaught axioms. In particular, we only need to show that the class is closed under increasing continuous chains. 

Let $\delta$ be a limit ordinal and $\{G_i \in \hat{K} : i < \delta\}$ be an increasing continuous chain. Let $G_\delta:= \bigcup_{i<\delta} G_i$. We show $G_\delta \in \hat{K}$ by considering two cases:
    
    \underline{Case 1}: For every $i < \delta$, $G_i \in \hat{K}_1$. We show that $G_\delta \in \hat{K}_1$. Clearly $p_1^\omega G_\delta \subseteq G_\delta$, so let $g\in G_\delta$. Then there is some $i<\delta$ such that $g\in G_i$. Since $p_1^\omega G_i=G_i$, we have that $g \in p_1^\omega G_i \subseteq p_1^\omega G_\delta$. Thus $p_1^\omega G_\delta=G_\delta$.
    
    Consider $p\neq p_1$, and suppose by way of contradiction that $p^\omega G_\delta \neq 0$. Then there exists  $g\in G_\delta$ such that $G_\delta \models p^n | g$ for every $n$. Since $g\in G_\delta$, there is some $i<\delta$ such that $g\in G_i$. Since $G_i \leq_p G_\delta$, we have that $G_i \models p^n | g$ for every $n$, so $p^\omega G_i \neq 0$, a contradiction as $G_i \in \hat{K}_1$. Hence $G_\delta \in \hat{K}_1$.

    \underline{Case 2}:  There is $i < \delta$ such that $G_i \notin \hat{K}_1$.  We show that $G_\delta \in \hat{K}_2$. Let $i_0 = \operatorname{min}\{ i < \delta : G_i \notin \hat{K}_1 \}$. Observe that for every $i \geq i_0$, $G_i \in \hat{K}_2$ by Proposition~\ref{basic}(3). We check Conditions (2)(a) and (2)(b) of Definition \ref{class}.

    We check Condition (2)(a). Let $g\in p_1^\omega G_\delta$. If $z_1, z_2 \in p_3^\omega G_\delta$ such that $g+z_1, g+z_2 \in p_4^\omega G_\delta$, then there is some $i<\delta$ such that $g,z_1,z_2 \in G_i$ and $i \geq i_0$. Since $G_i \in \hat{K}_2$, there is at most one $z\in p_3^\omega G_i$ such that $g+z\in p_4^\omega G_i$, so $z_1=z_2$. Thus $G_\delta$ satisfies Condition (2)(a).
    
    Along a similar vein, given $g\in p_1^\omega G_\delta$, there exists some $i_0\leq i<\delta$ such that $g\in p_1^\omega G_i$. Since $G_i \in \hat{K}_2$, there are $k\in \Z_{>0}$ and $z\in p_3^\omega G_i$ such that $kg+z \in p_4^\omega G_i$. But then $z\in p_3^\omega G_\delta$ and $k g+z \in p_4^\omega G_\delta$, so $G_\delta$ satisfies Condition (2)(b) and $G_\delta \in \hat{K}_2$. \qedhere

\end{proof}

We next investigate closure under direct sums.

\begin{prop}\label{direct sums}\
\begin{enumerate}
\item $\Kh_1$ and $\Kh_2$ are closed under direct sums.
\item $\Kh$ is not closed under direct sums.
\end{enumerate}
\end{prop}
\begin{proof}
    Observe that for every $G, H$ abelian groups and $p$ a prime number, $p^\omega (G\oplus H) = p^\omega G \oplus p^\omega H$.
    \begin{enumerate}
        \item That $\hat{K}_1$ is closed under direct sums follows directly from the observation above, so we focus on $\hat{K}_2$. Let $G, H \in \hat{K}_2$. We check Conditions (2)(a) and (2)(b) of Definition \ref{class}.
        
        We check Condition (2)(a). Let $(g,h) \in p_1^\omega (G \oplus H) = p_1^\omega G \oplus p_1^\omega H$. Suppose we have $(y_1, z_1), (y_2, z_2) \in p_3^\omega G\oplus p_3^\omega H$ such that $(g+y_1, h+z_1), (g+y_2, h+z_2) \in p_4^\omega G \oplus p_4^\omega H$. Then $g+y_1, g+y_2 \in p_4^\omega G$. Since $G \in \hat{K}_2$, there is at most one $y \in p_3^\omega G$ such that $g+y \in p_4^\omega G$, so $y_1=y_2$. Similarly since $H\in \hat{K}_2$, $z_1=z_2$. Thus $G\oplus H$ satisfies Condition (2)(a).
        
        We check Condition (2)(b). Consider $(g,h) \in p_1^\omega (G\oplus H) = p_1^\omega G \oplus p_1^\omega H$. Since $G,H \in \hat{K}_2$, there are $k_1, k_2 \in \Z_{>0}$ and $y \in p_3^\omega G, z \in p_3^\omega H$ such that $k_1g+y \in p_4^\omega G, k_2h+z \in p_4^\omega H$. Note $k_2(k_1g+y) \in p_4^\omega G$ and $k_1(k_2h+z) \in p_4^\omega H$, so  $(k_1k_2g+k_2y, k_1k_2h+k_1z) = k_1k_2(g,h) + (k_2y, k_1z) \in p_4^\omega G\oplus p_4^\omega H$, $k_1k_2 \in \Z_{>0}$, and $(k_2y, k_1z) \in p_3^\omega G \oplus p_3^\omega H$, satisfying Condition (2)(b).
        
                \item Let $G \in \hat{K}_1$ with $G\ne 0$ and $H \in \hat{K}_2$ with $p_3^\omega H \neq 0$. This is possible by taking for example $G$ and $H$ to be the subgroups of $\mathbb{Q}$ generated by $\{ \frac{1}{p_1^n} : n < \omega \}$ and 
                $\{ \frac{1}{p_3^n} : n < \omega \}$ respectively. Suppose by
way of contradiction that $G \oplus H \in \hat{K}$. Let $0 \neq h \in p_3^\omega H$, then $(0, h) \in p_3^\omega (G \oplus H)$, so $G\oplus H \not\in \hat{K}_1$. Hence $G\oplus H \in \hat{K}_2$.

        Let $0 \neq g \in G$. Then $(g, 0)  \in G \oplus p_1^\omega H = p_1^\omega (G\oplus H)$. So there are $k\in \Z_{>0}$ and $(g', h') \in p_3^\omega (G\oplus H)$ such that $(kg+g', h') \in p_4^\omega (G\oplus H)$ as $G\oplus H \in \hat{K}_2$. Since $p_3^\omega (G\oplus H) = p_3^\omega G\oplus p_3^\omega H = 0\oplus p_3^\omega H$ as $G \in \hat{K}_1$, we have that $g'=0$. So $kg+g'=kg \in p_4^\omega G=0$ as $G \in \hat{K}_1$, so $kg=0$. Hence $g =0$, a contradiction as $g \neq 0$. \qedhere
    \end{enumerate}
\end{proof}

We obtain the second statement of Theorem \ref{maint}.
\begin{lemma}\label{lemma:JEP-NMM}
$\Kh$  has the joint embedding property and no maximal models.
\end{lemma}
\begin{proof} We show first a claim.\medskip

    \underline{Claim}:  If $G \in \hat{K}$, then there is $H \in \hat{K}_2$ such that $G \leq_p H$.

    \underline{Proof of Claim}: If $G\in \hat{K}_2$, take $H=G$. So let $G\in \hat{K}_1$. Since $G$ is a torsion-free abelian group, there is a maximal linearly independent set $\{g_i : i\in I\} \subseteq G$. Note that we have a canonical isomorphism $\Q\otimes_\Z G \cong \bigoplus_{i\in I} \Q g_i$ for the injective hull $\Q\otimes_\Z G$ of $G$, where we may formally treat the $g_i$ on the right-hand side as free generators.    
    Match $\{g_i : i\in I\}$ with a set of free generators $\{z_i : i\in I\}$ and let $M := \Q \otimes_\Z G \oplus \bigoplus_{i\in I} \Q z_i$. In light of our canonical isomorphism $\Q\otimes_\Z G \cong \bigoplus_{i\in I} \Q g_i$, it will be more convenient to view $M$ as $M = \bigoplus_{i\in I} \Q g_i \oplus \bigoplus_{i\in I} \Q z_i$ 
    with free generators $\mathcal{S} = \{g_i,z_i : i\in I\}$ and to embed $G\le \Q\otimes_\Z G$ as a subgroup of $M$ by identifying $1\otimes g_i \in  \Q\otimes_\Z G$ with $g_i \in M$.

    Let $H:= \langle G, p_3^{-n} z_i, p_4^{-n}(g_i+z_i) : i\in I, n<\omega\rangle \leq M$. As it will be useful later, note that $\mathcal{S} \subseteq H$, $\mathcal{S}$ is a maximal linearly independent set in $H$, and $g_i \neq 0$ for every $i \in I$. 

    We next check $G \leq_p H$ by showing that $qG = qH \cap G$ for every prime $q$. Note that we only need to check $qH \cap G \subseteq qG$. Let $h \in qH \cap G$. In particular, $h=qh_0$ for some $h_0 \in H$. By definition of $H$, $h_0 = g + \sum_{i\in I} r_i z_i + \sum_{i\in I} s_i(g_i+z_i)$ for suitable $g\in G$, $r_i \in \Z[1/p_3]$, and $s_i \in \Z[1/p_4]$. Now $qh_0 = qg + \sum_{i\in I} qr_i z_i + \sum_{i\in I} qs_i(g_i+z_i) \in G \leq \bigoplus_{i\in I} \Q g_i$. Comparing coefficients of $z_i$ yields $r_i+s_i=0$ for all $i\in I$, so $s_i = -r_i \in \Z[1/p_3] \cap \Z[1/p_4] = \Z$. Thus $h_0 = g + \sum_{i\in I} (r_i z_i + s_i(g_i+z_i)) = g+\sum_{i\in I} s_i g_i \in G$ as each $s_i g_i \in G$.   
    
    It remains to prove that $H\in \hat{K}_2$.

We check    Condition (2)(a). Let $h \in p_1^\omega H$ be arbitrary, and suppose we have $z,z' \in p_3^\omega H$ such that $h+z, h+z' \in p_4^\omega H$. Then $z-z' \in p_3^\omega H \cap p_4^\omega H$. To conclude the proof of Condition (2)(a), we prove that $p_3^\omega H \cap p_4^\omega H = 0$.

    Let $h \in p_3^\omega H \cap p_4^\omega H$. Since $h\in H$ and $\mathcal{S}$ is a maximal linearly independent set in $H$, there are $n \neq 0, a_i, b_i \in \Z$ such that $nh = \sum_{i \in I} a_i g_i + \sum_{i \in I} b_i z_i$. But since $nh \in p_3^\omega H \cap p_4^\omega H$, we must have that $p_3^\infty p_4^\infty | \sum_{i \in I} (a_ig_i+b_iz_i)$. Since there are no relations between generators of different indices, we then have that $p_3^\infty p_4^\infty | (a_ig_i+b_iz_i)$ for all $i\in I$.

    Note that $p_3^\infty | z_i$, so $p_3^\infty | (a_i g_i + b_i z_i - b_i z_i)= a_i g_i$, and hence either $a_i=0$ or $p_3^\infty | g_i$. But $g_i \in G\leq_p H$ and $p_3^\omega G=0$, so if $a_i \neq 0$, we have $g_i=0$, a contradiction. Thus $a_i=0$.

    Now $p_3^\infty p_4^\infty | (a_ig_i+b_iz_i)$ means $p_3^\infty p_4^\infty | b_i z_i$. Note $p_4^\infty | (g_i+z_i)$, so $p_4^\infty | (b_i z_i - b_i(g_i+z_i))=-b_ig_i$, and hence either $b_i=0$ or $p_4^\infty | g_i$. But $ g_i \in G\leq_p H$ and $p_4^\omega G=0$, so if $b_i \neq 0$, we have that $g_i=0$, a contradiction. Thus $b_i=0$. Hence every $a_i, b_i = 0$, so $nh=0$ and $h=0$. So $p_3^\omega H \cap p_4^\omega H = 0$ and $H$ satisfies Condition (2)(a).

   We check Condition (2)(b). Let $h \in p_1^\omega H$ be arbitrary. Then there exist $k\in \Z_{>0}$ and $a_i, b_i \in \Z$ such that $kh = \sum_{i \in I} a_i g_i + \sum_{i \in I} b_i z_i$. Let $z = \sum_{i \in I} (a_i - b_i)z_i\in p_3^\omega H$. Then $kh+z = \sum_{i \in I} a_i(g_i+z_i) \in p_4^\omega H$, so $H$ satisfies Condition~(2)(b). Hence $H \in \hat{K}_2$. $\dagger_{\text{Claim}}$ \medskip

We show the joint embedding property. No maximal models can be shown with a similar argument. Let $G, H \in \hat{K}$. Then by the Claim there are $G_1 \in \hat{K}_2$ and $H_1 \in \hat{K}_2$ such that $G \leq_p G_1$ and $H \leq_p H_1$. As $\hat{K}_2$ is closed under direct sums by Proposition \ref{direct sums}, $G_1 \oplus H_1 \in \hat{K}_2$. Hence $G, H$ can be purely embed into $G_1 \oplus H_1 \in \hat{K}$. \end{proof}

In order to better understand Galois types, we show that $\Kh$ admits intersections and that the closure $\Cl_{\Kh}^G(A):= \bigcap \{H : A\subseteq H \leq_p G$ and $H\in \hat{K} \}$ can be constructed in a simple way.
\begin{prop}\label{closure}
   Let $A \subseteq G \in \hat{K}$.
   
   \begin{enumerate}
   \item If there is $H' \in \hat{K}_1$ such that $A \subseteq H' \leq_p G$, then $\Cl_{\Kh}^G(A) = A_0 \cup A_1 \cup A_2 = A_2 \in \hat{K}_1$ and $\| \Cl_{\Kh}^G(A) \| \leq |A| + \aleph_0 $ where
   \begin{itemize}
   \item $A_0 =A$,
   \item $A_1 = \langle A\rangle$,
   \item  $A_2 = \{g\in G : kg \in A_1 \mbox{ for some } k \in \Z_{>0}\}$.
   \end{itemize}
   
   \item If there is no $H' \in \hat{K}_1$ such that $A \subseteq H' \leq_p G$, then $\Cl_{\Kh}^G(A) = \bigcup_{n<\omega} A_n \in \hat{K}_2$ and $\| \Cl_{\Kh}^G(A) \| \leq |A| + \aleph_0 $ where
   
   \begin{itemize}
        \item $A_0 = A$,
        \item $A_{3n+1} = \langle A_{3n} \rangle$,
        \item $A_{3n+2} = \{g\in G :  kg \in A_{3n + 1} \mbox{ for some } k \in \Z_{>0}\}$,
        \item $A_{3n+3} =  A_{3n+2} \cup \{z_g : g\in A_{3n+2} \cap p_1^\omega G\}$ for a fixed choice of elements $z_g \in p_3^\omega G$ such that there is $k_g \in \Z_{>0}$ with $k_g g + z_g \in p_4^\omega G$.\footnote{Observe $z_g$ exists by Condition (2)(b) of Definition \ref{class} as $G \in \hat{K}_2$ and $g\in A_{3n+2} \cap p_1^\omega G$.}
   
    \end{itemize}
   \end{enumerate}
   
   In particular, $\Kh$ admits intersections and $\LS(\Kh)=\aleph_0$.

\end{prop}
\begin{proof}
The \emph{in particular} is clear.
\begin{enumerate}
    \item   Note that $H,H'\leq_p G$ for torsion-free groups implies $H\cap H' \leq_p H'$. Thus, we may assume without loss of generality that $G \in \hat{K}_1$ as $\bigcap \{H : A\subseteq H \leq_p G$ and $H\in \hat{K} \} = \bigcap \{H : A\subseteq H \leq_p H'$ and $H \in \hat{K}_1 \}$ by Proposition~\ref{basic}(2). It is straightforward to show that  $A \subseteq A_2 \leq_p G$, $A_2 \in \hat{K}_1$ and $\| A_2\| \leq |A| + \aleph_0$. 
    
     It follows from the previous paragraph that $\bigcap \{H : A\subseteq H \leq_p G$ and $H\in \hat{K} \} \subseteq A_2$. The other inclusion follows from the fact that $A_2 \subseteq H$ for every $ H \leq_p G$ with $A \subseteq H$ by uniqueness of divisors in torsion-free abelian groups.

    \item  Observe that $G \in \hat{K}_2$. Let $A_\omega = \bigcup_{n < \omega} A_n$. It can be shown that $A \subseteq A_\omega \leq_p G$, $\| A_\omega \| \leq |A| + \aleph_0$ and  $A_\omega \in \hat{K}_2$. Condition (2)(b) is satisfied due to the $A_{3n+2}$ construction step while Condition (2)(a) is inherited from $G$.

    We show $A_\omega = \bigcap \{H : A\subseteq H \leq_p G$ and $H\in \hat{K} \}$.  We only need to show the forward inclusion by the previous paragraph.
    
Fix $H$ such that $A \subseteq H \leq_p G$. We show by induction on $n <\omega$ that $A_n \subseteq H$. The base case and the case when $n=3m +1$ are clear so we do the other two cases.

Assume $n=3m+2$. Let $kg \in A_{3m+1}$ for $k \in \Z_{>0}$. Since $A_{3m+1} \subseteq H$ by induction hypothesis and $H \leq_p G$, there is $h \in H$ such that $kh=kg$. As $G$ is torsion-free, $g=h \in H$. Hence $A_{3m+2} \subseteq H$.

Assume  $n=3m+3$.  Let $g \in A_{3m+2} \cap p_1^\omega G$. Since $A_{3m+2} \subseteq H$ by induction hypothesis and  $H \leq_p G$,  $g\in p_1^\omega H$. Since $H \in \hat{K}_2$, there are $k \in \Z_{> 0}$ and $z \in p_3^\omega H$ such that $kg+z \in p_4^\omega H$, so $z\in p_3^\omega G$ and $kg+z \in p_4^\omega G$. Note that by construction, we have $z_g \in p_3^\omega G$ and $k_g g+z_g \in p_4^\omega G$. Hence $k_g kg + k_g z, k_gk  g + k z_g \in p_4^\omega G$. Then $k_g z = k z_g$ as $G \in \hat{K}_2$.

  As $k_g z \in k G$ and $k_g z \in H \leq_p G$, there is $h \in H$ such that $kh=k_gz$. Then $kh=kz_g$. Hence $z_g = h \in H$ because $G$ is torsion-free. Hence $A_{3m+3} \subseteq H$. \qedhere  
    \end{enumerate}
\end{proof}

The following classes were introduced in  \cite[Definition 3.1]{vaseyu}. 

\begin{defin} Let $\Kk$ be an AEC. We say that $\Kk$ is \emph{pseudo-universal} if $\Kk$ admits intersections and for any $N_1, N_2 \in \Kk$ and $\bar{a}_1 \in N_1$, $\bar{a}_2 \in N_2$ if $f,g: \Cl^{N_1}_{\Kk}(\bar{a}_1) \cong \Cl^{N_2}_{\Kk}(\bar{a}_2)$ and $f(\bar{a}_1)=g(\bar{a}_1)=\bar{a}_2$, then $f=g$.\footnote{The definition in \cite{vaseyu} has the additional assumption in the antecedent that $\gtp_{\Kk}(\bar{a}_1/\emptyset; N_1) = \gtp_{\Kk}(\bar{a}_2/\emptyset; N_2)$, but this already follows from the existence of $f,g$.}
\end{defin}

We obtain the last statement of Theorem \ref{maint}.

\begin{theorem}\label{Thm:tame}
    $\Kh$ is pseudo-universal. In particular, $\Kh$ is $(<\aleph_0)$-tame. 
\end{theorem}
\begin{proof}

The \emph{in particular} follows from \cite[Theorem 3.7]{vaseyu}. We show that $\Kh$ is pseudo-universal.

    Suppose  $f,g : \Cl_{\Kh}^{G_1}(\bar{a}_1) \cong \Cl_{\Kh}^{G_2}(\bar{a}_2)$ with $f(\bar{a}_1)=g(\bar{a}_1)=\bar{a}_2$. We show that $f=g$. We divide the proof into two cases based on Proposition \ref{closure}.
    
    \underline{Case 1}: There is an $H' \in \hat{K}_1$ such that $\bar{a}_1 \in  H' \leq_p G_1$. We show that ${f\restr A_n} = g\restr A_n$ for $n \in \{0, 1, 2\}$, where $A_0 =\{\bar{a}_1\}$. It is clear that $ f\restr A_0 = g\restr A_0$  and  $f\restr A_1 = g\restr A_1$.

    Let $ka \in A_1$   for $k \in \Z_{>0}$. Then $f(ka) = g(ka)$ as  $f\restr A_1 = g\restr A_1$. Since $f, g$ are morphisms, $ kf(a)=kg(a)$. Hence $f(a)=g(a)$ because $G_2$ is torsion-free. Hence $f\restr A_2 = g\restr A_2$.
        
    \underline{Case 2}: There is no $H' \in \hat{K}_1$ such that $\bar{a}_1 \in H' \leq_p G_1$.  We show that  $f\restr A_n = g\restr A_n$ by induction on $n < \omega$, where $A_0 =\{\bar{a}_1\}$. 
        
        The cases when $n=0, 3m+1, 3m+2$ can be done as in Case 1, so we only do the induction step when $n=3m+3$.

        Let $z_a \in A_{3m+3}$ be such that $a\in A_{3m+2} \cap p_1^\omega G_1$, $k_a a+z_a \in p_4^\omega G_1$, $z_a \in p_3^\omega G_1$ and $k_a \in \Z_{>0}$. We may assume that $z_a \neq 0$ as otherwise $f(z_a)=0=g(z_a)$.
        
       Observe that $f(k_aa+z_a) = k_af(a)+f(z_a)  \in p_4^\omega G_2$ and $f(z_a) \in p_3^\omega G_2$ since $a, z_a, k_aa+z_a \in \Cl_{\Kh}^{G_1}(\bar{a}_1) \leq_p G_1$ and $f$ is a morphism. Similarly  $g(k_aa+z_a)=k_ag(a)+g(z_a) \in p_4^\omega G_2$ and $g(z_a) \in p_3^\omega G_2$. Since $g(a) = f(a)$ by induction hypothesis, we have that $k_af(a)+g(z_a) \in p_4^\omega G_2$. As $G_2 \in \hat{K}_2$ because $G_2 \in \hat{K}$ and $0\neq g(z_a) \in p_3^\omega G_2$, it follows by Condition (2)(a) that $f(z_a) = g(z_a)$. Hence $f \restr A_{3m+3} = g\restr A_{3m+3}$. \qedhere
\end{proof}
    
\begin{remark} It is unclear to us if the example of 
\cite{ps} is $(<\aleph_0)$-tame or has the joint embedding property.
\end{remark}

  We now turn to show that $\Kh$ is not stable. In order to do that we will need to construct some specific abelian groups and show that they are in $\Kh$. 
  
\begin{defin}\label{groups}
    Let $\lambda$ be an infinite cardinal, let $\{x_\alpha, z_\alpha : \alpha < \lambda\}$ be a set of free generators, and let $p_5$ be a prime distinct to $p_1, p_3, p_4$. 
    
    \begin{enumerate}
        \item Let $M_\lambda:= \bigoplus_{\alpha<\lambda} \Q x_\alpha  \oplus  \bigoplus_{\alpha<\lambda} \Q z_\alpha$.
        \item  Let $G_\lambda := \left\langle p_1^{-n} x_\alpha : n<\omega, \alpha < \lambda \right\rangle_{M_\lambda} \leq M_\lambda$.
        
        \item If $\mcU \subseteq \lambda$, let  \[G_\mcU := \big\langle p_1^{-n} x_\alpha, p_3^{-n} z_\alpha, p_4^{-n}(x_\alpha+z_\alpha), p_5^{-n} z_\beta : n<\omega, \alpha<\lambda, \beta\in \mcU \big\rangle_{M_\lambda}.\]
    \end{enumerate}
\end{defin}

\begin{prop}\label{div-G-U}
    Let $\lambda$ be an infinite cardinal and $\mcU \subseteq \lambda$.

    \begin{enumerate}
        \item $p_1^\omega G_\mcU = \langle p_1^{-n} x_\alpha :  n <\omega, \alpha<\lambda \rangle =G_\lambda$.
        \item $p_3^\omega G_\mcU = \langle p_3^{-n} z_\alpha, p_5^{-n} z_\beta :  n <\omega, \alpha<\lambda, \beta \in \mcU \rangle$.
        \item $p_4^\omega G_\mcU = \langle p_4^{-n} (x_\alpha+z_\alpha) :  n <\omega, \alpha<\lambda \rangle$.
        \item $p_5^\omega G_\mcU = \langle p_3^{-n} z_\beta, p_5^{-n} z_\beta :  n <\omega, \beta \in \mcU \rangle$.
    \end{enumerate}
\end{prop}
\begin{proof} 
    Observe that for any $n<\omega$ and $\alpha<\lambda$, $p_3^{-n} z_\alpha \in p_3^\omega G_\mcU$ since for any $k<\omega$, $p_3^{-n} z_\alpha = p_3^k (p_3^{-n-k} z_\alpha)$. Moreover, for any $\beta \in \mcU$, choosing $s,t \in \Z$ with $sp_3^k+tp_5^n =1$, we have $p_5^{-n} z_\beta = p_5^{-n} (sp_3^k+tp_5^n) z_\beta = p_3^k (sp_5^{-n}z_\beta+tp_3^{-k} z_\beta)$. Thus $p_3^\omega G_\mcU \supseteq \langle p_3^{-n} z_\alpha, p_5^{-n} z_\beta :  n <\omega, \alpha<\lambda, \beta \in \mcU \rangle$. The other three cases of $p_i^\omega G_\mcU$ are similar. Thus it remains to prove the reverse containment.

    Let $g \in G_\mcU$ be arbitrary. Since $G_\mcU \leq M_\lambda = \bigoplus_{\alpha<\lambda} \Q x_\alpha \oplus \bigoplus_{\alpha<\lambda} \Q z_\alpha$, we can express
    \[
        g = \sum_{\alpha <\lambda} r_\alpha x_{\alpha} + \sum_{\alpha <\lambda} s_\alpha z_{\alpha}
    \]
    as a finite sum where $r_\alpha, s_\alpha \in \Q$ for all $\alpha<\lambda$ and almost all $r_\alpha, s_\alpha$ are zero. Since $g\in G_\mcU$, we can actually express more precisely what each of the $r_\alpha, s_\alpha$ are. First note that we can write
    \[
        g = \sum_{\alpha <\lambda} \Big( p_1^{-n_1} i_\alpha x_\alpha + p_3^{-n_3} j_\alpha z_\alpha + p_4^{-n_4} k_\alpha(x_\alpha + z_\alpha) + p_5^{-n_5} \ell_\alpha z_\alpha \Big),
    \]
    where $n_1, n_3, n_4, n_5 < \omega$ and $i_\alpha, j_\alpha, k_\alpha, \ell_\alpha \in \Z$ with almost all $i_\alpha, j_\alpha, k_\alpha, \ell_\alpha$ zero. Note that $\ell_\alpha = 0$ for every $\alpha \not\in \mcU$. Then all $r_\alpha$ and $s_\alpha$ are of the following form:
    \begin{align*}
        r_\alpha &= p_1^{-n_1} i_{\alpha} + p_4^{-n_4} k_{\alpha}, \\
        s_\alpha &= p_3^{-n_3} j_{\alpha} + p_4^{-n_4} k_{\alpha} + p_5^{-n_5} \ell_{\alpha}.
    \end{align*}

    Observe that the above lines actually prove that the coefficients of $x_\alpha$ are elements of $\Z[1/p_1, 1/p_4]$, and the coefficients of $z_\alpha$ are elements of $\Z[1/p_3, 1/p_4, 1/p_5]$ for all $\alpha<\lambda$. More subtly, $s_\alpha\in \Z[1/p_3, 1/p_4]$ if $\alpha \not\in \mcU$. We also have $r_\alpha-s_\alpha = p_1^{-n_1} i_{\alpha} -p_3^{-n_3} j_{\alpha} - p_5^{-n_5} \ell_{\alpha}$, thus $r_\alpha-s_\alpha \in \Z[1/p_1, 1/p_3, 1/p_5]$ for all $\alpha<\lambda$.

    \begin{enumerate}
        \item Suppose $g\in p_1^\omega G_\mcU$. Then $p_1^\infty | g$, so we have that $s_\alpha$ is $p_1^\infty$-divisible in $\Z[1/p_3, 1/p_4, 1/p_5]$, so $s_\alpha=0$. Thus
        \[
            p_4^{-n_4} k_{\alpha} = -p_3^{-n_3} j_{\alpha} - p_5^{-n_5} \ell_{\alpha} \in \Z[1/p_4] \cap \Z[1/p_3, 1/p_5] = \Z.
        \]
        Thus $r_\alpha = p_1^{-n_1} i_{\alpha} + c_\alpha$ for $c_\alpha = p_4^{-n_4} k_{\alpha} \in \Z$. Thus $r_\alpha \in \Z[1/p_1]$, so $g\in \langle p_1^{-n} x_\alpha :  n <\omega, \alpha<\lambda \rangle$, concluding the proof of (1).
        
        \item The proof of (2) is similar to that of (1) after interchanging the roles of $r_\alpha$ and $s_\alpha$ in the proof, and is hence omitted. In particular, $r_\alpha=0$ and $p_4^{-n_4} k_{\alpha} \in \Z[1/p_1] \cap \Z[1/p_4] =\Z$.       
        
        \item Suppose $g\in p_4^\omega G_\mcU$. Then $p_4^\infty | g$, so we have that $r_\alpha-s_\alpha$ is $p_4^\infty$-divisible in $\Z[1/p_1, 1/p_3, 1/p_5]$, so $r_\alpha-s_\alpha=0$. Then        
        %
        $r_\alpha = s_\alpha \in \Z[1/p_1, 1/p_4] \cap \Z[1/p_3, 1/p_4, 1/p_5] = \Z[1/p_4]$. Thus $g = \sum_{\alpha < \lambda} r_\alpha x_{\alpha} +\sum_{\alpha < \lambda} s_\alpha z_{\alpha} = \sum_{\alpha < \lambda} s_\alpha (x_{\alpha}+ z_{\alpha}) \in \langle p_4^{-n} (x_\alpha+z_\alpha) :  n <\omega, \alpha<\lambda \rangle$, completing the proof of (3).
        
        \item The proof of (4) is similar to that of (1), and is hence omitted. In particular, $r_\alpha=0$, $s_\alpha=0$ for every $\alpha \not\in \mcU$, and $p_4^{-n_4} k_{\alpha} \in \Z[1/p_1] \cap \Z[1/p_4] =\Z$.
        \qedhere
        
        
        
        
    \end{enumerate}
\end{proof}

\begin{prop}\label{in-Kh}
Let $\lambda$ be an infinite cardinal and $\mcU \subseteq \lambda$.
\begin{enumerate}
\item $G_\lambda \in \hat{K}_1$ and $\| G_\lambda\|= \lambda$.
\item $G_\mcU \in \hat{K}_2$ and $\| G_\mcU\|= \lambda$.
\item $G_\lambda \leq_p G_\mcU$. 
\end{enumerate}
\end{prop}
\begin{proof}  \
\begin{enumerate}
    \item Let $g = \sum_{\alpha<\lambda} r_\alpha x_\alpha \in G_\lambda$ for $r_\alpha \in \Z[1/p_1]$ and $n<\omega$. Then $g' := \sum_{\alpha<\lambda} p_1^{-n}r_\alpha x_\alpha \in G_\lambda$ and $g = p_1^n g'$, so $G_\lambda = p_1^\omega G_\lambda$.

    Let $p \neq p_1$ be a prime, and let $g = \sum_{\alpha<\lambda} r_\alpha x_\alpha\in p^\omega G_\lambda$. Then $p^\infty \mid g$, so every $r_\alpha$ is $p$-divisible in $\Z[1/p_1]$, so $r_\alpha=0$ and $g=0$.    
    %
    Thus $G_\lambda \in \hat{K}_1$.
    
    Note that $|G_\lambda|$ is the set of all finite linear combinations of $\{p_1^{-n} x_\alpha : n< \omega, \alpha < \lambda\}$, so $\|G_\lambda\|=\lambda$.
    
    \item 
    We first show  a claim.\medskip
    
    \underline{Claim}: 
    $p_3^\omega G_\mcU \cap p_4^\omega G_\mcU=0$.
    

    \underline{Proof of Claim}: Let $g \in p_3^\omega G_\mcU \cap p_4^\omega G_\mcU$. Then $g = \sum_{\alpha<\lambda} s_\alpha z_\alpha = \sum_{\alpha<\lambda} t_\alpha (x_\alpha+z_\alpha)$ for suitable $s_\alpha, t_\alpha \in \Q$ by Proposition \ref{div-G-U}. Comparing coefficients of $x_\alpha$, we find $t_\alpha=0$ for all $\alpha <\lambda$, so $g = \sum 0(x_\alpha+z_\alpha) = 0$, and $p_3^\omega G_\mcU \cap p_4^\omega G_\mcU = 0$. $\dagger_{\text{Claim}}$\medskip

    We show $G_\mcU \in \hat{K}_2$. Let $g\in p_1^\omega G_\mcU$. If we have $z,z' \in p_3^\omega G_\mcU$ such that $g+z, g+z' \in p_4^\omega G_\mcU$, then $z-z' \in p_3^\omega G_\mcU \cap p_4^\omega G_\mcU = 0$, so $z=z'$, satisfying Condition (2)(a).

    Consider $g\in p_1^\omega G_\mcU$. Then $g = \sum_{\alpha<\lambda} p_1^{-n} k_\alpha x_\alpha$ for some $n<\omega$, $k_\alpha \in \Z$ all but finitely many zero. Let $z = \sum_{\alpha<\lambda} k_\alpha z_\alpha \in p_3^\omega G_\mcU$, and note
    \[
    p_1^n g + z = \sum_{\alpha<\lambda} k_\alpha(x_\alpha+z_\alpha) \in p_4^\omega G_\mcU.
    \]
    Thus $G_\mcU$ satisfies Condition (2)(b), so $G_\mcU \in \hat{K}_2$.

    To see $\|G_\mcU\|=\lambda$, note $|G_\mcU|$ is the set of all finite linear combinations of $\{p_1^{-n} x_\alpha, p_3^{-n} z_\alpha, p_4^{-n}(x_\alpha+z_\alpha), p_5^{-n} z_\beta : n< \omega, \alpha < \lambda, \beta \in \mcU\}$, so $\|G_\mcU\|=\lambda$.
    
    \item To prove $G_\lambda \leq_p G_\mcU$, we show $G_\lambda \cap q G_\mcU = q G_\lambda$ for any prime $q$. Let $g \in G_\lambda \cap q G_\mcU$. Then there exist $k_\alpha, k_{1,\alpha}, k_{3,\alpha}, k_{4,\alpha}, k_{5,\alpha} \in \Z$ almost all zero, and $m_1, m_3, m_4, m_5 \in \Z_{\geq 0}$ such that
    \begin{align*}
        g &= \sum_{\alpha < \lambda} p_1^{-m_1} k_\alpha x_\alpha \\
        &= q \sum_{\alpha < \lambda} \Big(p_1^{-m_1}  k_{1,\alpha} x_\alpha + p_3^{-m_3} k_{3,\alpha} z_\alpha + p_4^{-m_4} k_{4,\alpha} (x_\alpha+z_\alpha) + p_5^{-m_5} k_{5,\alpha} z_\alpha \Big).
    \end{align*}
    Multiplying everything by $p_1^{m_1} p_3^{m_3} p_4^{m_4} p_5^{m_5}$ yields
    \begin{align*}
        p_1^{m_1} p_3^{m_3} p_4^{m_4} p_5^{m_5} g =\sum_{\alpha < \lambda} & p_3^{m_3} p_4^{m_4} p_5^{m_5} k_\alpha x_\alpha \\
        = \sum_{\alpha < \lambda} & \Big( q p_3^{m_3} p_4^{m_4} p_5^{m_5} k_{1,\alpha} x_\alpha
        + q p_1^{m_1} p_4^{m_4} p_5^{m_5} k_{3,\alpha} z_\alpha \\
        + & q p_1^{m_1} p_3^{m_3} p_5^{m_5} k_{4,\alpha} (x_\alpha+z_\alpha) + q p_1^{m_1} p_3^{m_3} p_4^{m_4} k_{5,\alpha} z_\alpha \Big).
    \end{align*}

    Comparing the coefficients of the $z_\alpha$ and canceling out a $q p_1^{m_1}$, we see that $0 = p_4^{m_4} p_5^{m_5} k_{3,\alpha} + p_3^{m_3} p_5^{m_5} k_{4,\alpha} + p_3^{m_3} p_4^{m_4} k_{5,\alpha}$. Thus for every $\alpha < \lambda$, $p_3^{m_3} | k_{3,\alpha}$, $p_4^{m_4} | k_{4,\alpha}$, and $p_5^{m_5} | k_{5,\alpha}$ in $\Z$. So there exist $r_{3,\alpha}, r_{4,\alpha}, r_{5,\alpha} \in \Z$ such that $k_{3,\alpha} = p_3^{m_3} r_{3,\alpha}$, $k_{4,\alpha} = p_4^{m_4} r_{4,\alpha}$, and $k_{5,\alpha} = p_5^{m_5} r_{5,\alpha}$. Substituting these into our expression for $p_1^{m_1} p_3^{m_3} p_4^{m_4} p_5^{m_5} g$ and canceling $p_3^{m_3} p_4^{m_4} p_5^{m_5}$ from every term yields
    \begin{align*}
        p_1^{m_1}g &=\sum_{\alpha < \lambda} k_\alpha x_\alpha \\
        &= \sum_{\alpha < \lambda} \Big( q k_{1,\alpha} x_\alpha + q p_1^{m_1} r_{3,\alpha} z_\alpha + q p_1^{m_1} r_{4,\alpha} (x_\alpha+z_\alpha) + q p_1^{m_1} r_{5,\alpha} z_\alpha \Big).
    \end{align*}
      Comparing the coefficients of the $x_\alpha$, we get that $q| k_\alpha$ in $\Z$ for all $\alpha < \lambda$, so there exist $r_\alpha \in \Z$ such that $k_\alpha = q r_\alpha$, and $g = \sum_{\alpha < \lambda} p_1^{-m_1}k_\alpha x_\alpha = q\sum_{\alpha < \lambda} p_1^{-m_1}r_\alpha x_\alpha \in q G_\lambda$. \qedhere
\end{enumerate}
\end{proof}

We obtain the first statement of Theorem \ref{maint}. 

\begin{theorem}\label{Prop:K_not_stable}
For every infinite cardinal $\lambda$, $\Kh$ is not stable in $\lambda$. 
\end{theorem}
\begin{proof}
We first show a  claim.\medskip

\underline{Claim}: For every $\alpha < \lambda$, $z_\alpha \in \Cl_{\Kh}^{G_\mcU}(\{ z_0\} \cup G_\lambda)$.
        
        \underline{Proof of Claim}: Let $\alpha < \lambda$. Observe that  
      $\Cl_{\Kh}^{G_\mcU}(\{ z_0\} \cup G_\lambda)   \in \hat{K}_2$ as $z_0 \in p_3^\omega \Cl_{\Kh}^{G_\mcU}(\{ z_0\} \cup G_\lambda)$. So there is an $m \in \Z_{> 0}$  and a $z' \in p_3^\omega \Cl_{\Kh}^{G_\mcU}(\{ z_0\} \cup G_\lambda)$ such that $mx_\alpha +z' \in p_4^\omega \Cl_{\Kh}^{G_\mcU}(\{ z_0\} \cup G_\lambda)$. Then $z' \in p_3^\omega G_\mcU$ and $mx_\alpha + z' \in p_4^\omega G_\mcU$. As $z_\alpha \in p_3^\omega G_\mcU$ and $mx_\alpha + mz_\alpha \in p_4^\omega G_\mcU$ by definition of $G_\mcU$, it follows from Condition (2)(b) that $mz_\alpha =z' \in \Cl_{\Kh}^{G_\mcU}(\{ z_0\}\cup G_\lambda)$. Hence $z_\alpha \in \Cl_{\Kh}^{G_\mcU}(\{ z_0\} \cup G_\lambda)\leq_p G_\mcU$ because $G_\mcU$ is torsion-free. $\dagger_{\text{Claim}}$\medskip

It is enough to show that  if  $\mcU, \mcV \subseteq \lambda$ with $\mcU \neq \mcV$, then $\gtp_{\Kh}(z_0/G_\lambda; G_\mcU) \neq \gtp_{\Kh}(z_0/G_\lambda; G_\mcV)$ as this guarantees $|\gS_{\Kh}(G_\lambda)|\ge 2^\lambda > \lambda = \|G_\lambda\|$.
Thus, suppose by way of contradiction that there are $\mcU \neq \mcV$ such that $\gtp_{\Kh}(z_0/G_\lambda; G_\mcU) = \gtp_{\Kh}(z_0/G_\lambda; G_\mcV)$. Without loss of generality, let $\beta \in \mcU \setminus \mcV$.

Since $\Kh$ admits intersections, there exists an isomorphism $f:\Cl_{\Kh}^{G_\mcU}(\{ z_0\} \cup G_\lambda) \cong \Cl_{\Kh}^{G_\mcV}
(\{ z_0\} \cup G_\lambda)$ such that $f(z_0)=z_0$ and $f \restr {G_\lambda} = \id_{G_\lambda}$. 

We first show that $f(z_\alpha)= z_\alpha$ for every $\alpha < \lambda$.  First observe that $f(z_\alpha)$ is well-defined as $z_\alpha \in  \Cl_{\Kh}^{G_\mcU}(\{ z_0\} \cup G_\lambda)$ by the claim. Moreover, $f(z_\alpha) \in p_3^\omega G_\mcV$ and $f(x_\alpha + z_\alpha) =x_\alpha + f(z_\alpha) \in p_4^\omega G_\mcV$ as $z_\alpha \in p_3^\omega G_\mcU$ and $x_\alpha + z_\alpha \in p_4^\omega G_\mcU$ by the definition of $G_\mcU$. Similarly, $z_\alpha \in p_3^\omega G_\mcV$ and $x_\alpha + z_\alpha \in p_4^\omega G_\mcV$ by the definition of $G_\mcV$. Hence $f(z_\alpha) = z_\alpha$ since $G_\mcV \in \hat{K}_2$.

As $\beta \in \mcU$, $z_\beta \in p_5^\omega G_\mcU$ by definition of $G_\mcU$. Hence $f(z_\beta) =  z_\beta \in p_5^\omega G_\mcV$ where $f(z_\beta)=z_\beta$ by the previous paragraph. This is a contradiction to $\beta \not\in \mcV$ and  Proposition \ref{div-G-U}(4).\end{proof}

\begin{remark}
The choice of $z_0$ in the proof of the last theorem is a rather superficial one. More generally, if $\mcU, \mcV \subseteq \lambda$ with $\mcU \neq \mcV$, $a \in G_\mcU \setminus G_\lambda$, and $b \in G_\mcV \setminus G_\lambda$, then $\gtp_{\Kh}(a/G_\lambda; G_\mcU) \neq \gtp_{\Kh}(b/G_\lambda; G_\mcV)$ holds by a similar argument.
\end{remark}

\begin{remark}
Observe that $\Kh$ is also an example of an AEC of modules such that Galois types are not pp-syntactic in the sense of \cite[Definition 3.6]{maz2}, but $\Kh$ is $(<\aleph_0)$-tame. Galois types are not pp-syntactic because $\Kh$ is not stable and \cite[Theorem 3.8]{maz2}. 
\end{remark}

\begin{remark}
  The Galois types that we use to show that $\Kh$ is not stable are \emph{essentially} the same as those of \cite[Section 4]{ps}. The only differences are that \cite{ps} has an additional generator $y$ and prime number $p_2$ attached to it and that $p_5$ has also some relation to $x_\alpha$, but these are unnecessary in both settings.    
\end{remark}

We obtain the third statement of Theorem \ref{maint}.

\begin{lemma}\label{Prop:K_no_AP}
    $\Kh$ does not have the amalgamation property.
\end{lemma}
\begin{proof}
For $\mcU, \mcV \subseteq \lambda$ with $\mcU \neq \mcV$, a similar argument to that of Theorem \ref{Prop:K_not_stable} can be used to show that $G_\lambda \leq_p G_\mcU, G_\mcV$ cannot be completed to a commutative square of pure embeddings. 
\end{proof}

\begin{remark}
More generally, using a similar argument as in the last proof, the family $\{G_\lambda, G_\mcU : \mcU \subseteq \lambda\}\subseteq \hat{K}$ satisfies the following properties:
\begin{itemize}
\item $\|G_\lambda\|=\lambda$.
\item $G_\lambda \leq_p G_\mcU$ for all $\mcU \subseteq \lambda$.
\item For any $\mcU, \mcV \subseteq \lambda$ with $\mcU \neq \mcV$, there does not exist any finite sequence of groups $H_i\in \hat{K}$ for $0\le i \le k$ with $H_0=G_\mcU$ and $H_k = G_\mcV$ such that $G_\lambda \leq_p H_i, H_{i+1}$ can be completed to a commutative square of pure embeddings for all $i<k$.
\end{itemize}
In the context of the standard definition of Galois types, the existence of such a family $\{G_\lambda, G_\mcU : \mcU \subseteq \lambda\}$ is sufficient for proving that $\Kh$ fails to be stable in $\lambda$.
\end{remark}

We finish by providing the proof of the main theorem of our paper.

 \begin{proof}[Proof of Theorem \ref{maint}]
$\Kh$ is an AEC with  $\LS(\Kh)=\aleph_0$ by Lemma \ref{kh-aec}.

\begin{enumerate}
    \item Follows from Theorem \ref{Prop:K_not_stable}.
    \item Follows from Lemma \ref{lemma:JEP-NMM}.
    \item Follows from Lemma \ref{Prop:K_no_AP}.
    \item Follows from Theorem \ref{Thm:tame}. \qedhere
\end{enumerate}
     \end{proof}

	\end{document}